\documentclass[intlimits]{article}
\usepackage{latexsym,amsfonts,amssymb,amsmath,amsthm}

\newtheorem{theorem}{Theorem}[section]

\newtheorem{lemma}[theorem]{Lemma}
\newtheorem{proposition}{Proposition}

\theoremstyle{definition}

\newtheorem{remark}{Remark}

\newcommand\RR{\ensuremath{\mathbb{R}}}

\newcommand\NN{\ensuremath{\mathbb{N}}}

\newcommand{\PP}{\ensuremath{\mathbf{P}}}

\newcommand{\OFP}{\ensuremath{(\Omega,\mathcal{F},\mathbf{P})}}

\newcommand{\norm}[1]{\ensuremath{\lVert#1\rVert}}

\DeclareMathOperator{\diag}{diag}

\DeclareMathOperator{\lnminus}{ln^{--}}
\DeclareMathOperator{\lnplus}{ln^{+}}

\begin{document}

\title{Averaging in Random Systems of Nonnegative Matrices}

\author{Janusz Mierczy\'nski
\\
Institute of Mathematics and Computer Science
\\
Wroc{\l}aw University of Technology
}
\date{}

\footnotetext{This is a pre-copy-editing, author-produced PDF of an article accepted for publication in \emph{DCDS Supplements} following peer review. The definitive publisher-authenticated version in: \emph{Dynamical Systems and Differential Equations, AIMS Proceedings 2015, Proceedings of the 10th AIMS International Conference (Madrid, Spain)}, pp. 835--840, is available online at: http://dx.doi.org/10.3934/proc.2015.0835. }

\maketitle

\begin{abstract}
It is proved that for the top Lyapunov exponent of a random matrix system of the form $\{A D(\omega)\}$, where $A$ is a nonnegative matrix and $D(\omega)$ is a diagonal matrix with positive diagonal entries, is bounded from below by the top Lyapunov exponent of the averaged system.  This is in contrast to what one should expect of systems describing biological metapopulations.
\end{abstract}

\section{Introduction}

We assume throughout that $\OFP$ is a probability space:  $\mathcal{F}$ is a $\sigma$\nobreakdash-\hspace{0pt}algebra of subsets of $\Omega$, and $\mathbf{P}$ is a probability measure defined on $\mathcal{F}$.

Let $\theta \colon \Omega \to \Omega$ be an automorphism of the probability space $\OFP$. We assume that $\theta$ is ergodic:  for any $\Omega' \in \mathcal{F}$ such that $\theta(\Omega') = \Omega'$ either $\PP(\Omega') = 1$ or $\PP(\Omega') = 0$.

\smallskip
By a {\em random matrix system\/} we understand a measurable family $S \colon \Omega \to \RR^{N \times N}$ of linear endomorphisms of $\RR^N$ (identified with $N$ by $N$ real matrices),
\begin{equation}
\label{eq:RMS}
\{ S(\omega) \}_{\omega \in \Omega}.
\end{equation}

A random matrix system gives rise to a discrete\nobreakdash-\hspace{0pt}time linear (skew\nobreakdash-\hspace{0pt}product) random (semi)dynamical system on $\Omega \times \RR^N$ consisting of iterates of the vector bundle morphism
\begin{equation*}
(\omega, u) \mapsto (\theta\omega, S(\omega)u), \quad \omega \in \Omega, \ u \in \RR^N.
\end{equation*}
Its $n$\nobreakdash-\hspace{0pt}th iterate, $n = 1, 2, 3, \dots$, has the form
\begin{equation*}
(\omega, u) \mapsto (\theta^{n}\omega, S^{(n)}(\omega)u), \quad \omega \in \Omega, \ u \in \RR^N,
\end{equation*}
where
\begin{equation*}
S^{(n)}(\omega) := S(\theta^{n-1}\omega) \dots S(\theta \omega) S(\omega), \quad n = 1, 2, 3, \dots, \ \omega \in \Omega.
\end{equation*}
The following result is a part of the Furstenberg--Kesten theorem (see, e.g., \cite[Thm.\ 3.3.3]{Arn}).
\begin{proposition}
\label{prop:top-Lyapunov}
For a random matrix system \eqref{eq:RMS} assume that the mapping $\lnplus{\norm{S(\cdot)}}$ belongs to $L_1\OFP$.  Then there exists $\lambda \in [-\infty, \infty)$ such that for a.e.\ $\omega \in \Omega$ the equality
\begin{equation}
\label{eq:Kingman1}
\lambda = \lim\limits_{n \to \infty} \frac{\ln{\norm{S^{(n)}(\omega)}}}{n}.
\end{equation}
Moreover,
\begin{equation}
\label{eq:Kingman2}
\lambda = \inf\limits_{n \in \NN} \frac{1}{n} \int\limits_{\Omega} \ln{\norm{S^{(n)}(\cdot)}} \, d\mathbf{P}(\cdot).
\end{equation}
\end{proposition}
(Here and in the sequel $\norm{\cdot}$ denotes the Euclidean matrix or vector norm, depending on the context.)

We will call $\lambda$ as above the {\em top Lyapunov exponent\/} of the random matrix system~\eqref{eq:RMS}.

\bigskip
In the present paper we consider random matrix systems of a special form, namely such that $S(\omega) = A D(\omega)$, where $A$ is a constant (that is, independent of $\omega$) matrix with nonnegative entries and $D(\omega)$ is a diagonal matrix with positive diagonal entries.

Such random matrix systems occur in modeling so\nobreakdash-\hspace{0pt}called metapopulations, that is, populations in which individuals live in $N$ spatially separated patches (see, e.g., \cite{Schreib}).  Here $u_i$, $1 \le i \le N$, is the number of individuals in patch $i$, $d_{i}$ is the fitness of an individual in patch $i$, and $a_{ij}$ is, for $i \ne j$, the fraction of the population from patch $j$ that disperse to patch $i$.

The top Lyapunov exponent measures the overall fitness of the metapopulation:  the larger it is the more viable the (meta)population should be.  Indeed, if $A$ is a primitive matrix (meaning that some of its powers has all entries positive) then the logarithmic growth rate of iterates of any positive vector equals the top Lyapunov exponent.

It is an important subject in population dynamics to analyze the influence of seasonal variations on the fitness.  To quote Sebastian J. Schreiber~\cite{Schreib}:
\begin{quote}
Temporal fluctuations in environmental conditions can lead to fluctuations in population growth rates. For a given mean population growth rate, one expects that
extinction risk increases with temporal variation in the growth rates.
\end{quote}
Let us look at the mathematical interpretation of the above statement in the language of random matrix systems of the form $S(\omega) = A D(\omega)$.  As the dynamical system generated by $\theta$ on the base space is ergodic, for each patch $i$, $1 \le i \le N$, Birkhoff's ergodic theorem states that for a.e.\ $\omega \in \Omega$ the limit
\begin{equation*}
\lim\limits_{n \to \infty} \frac{1}{n} \sum\limits_{k=0}^{n-1} \ln{d_i(\theta^{k}\omega)}
\end{equation*}
exists and equals the expected value
\begin{equation*}
\int\limits_{\Omega} \ln{d_i(\cdot)}\, d\PP(\cdot),
\end{equation*}
which can be interpreted as the mean population growth rate in isolated patch $i$.   As dispersal rates are independent of time, one compares the top Lyapunov exponent of the original system with the top Lyapunov exponent of the system with the population growth rate in each patch $i$ replaced by its geometric mean.  The latter Lyapunov exponent equals just the logarithm of the spectral radius of $A \bar{D}$, where $\bar{D}$ is the diagonal matrix obtained by taking the geometric means of the entries of $D$.  Therefore, our expectations should be that the top Lyapunov exponent of the system with temporal variation is not larger than the metapopulation growth rate for the averaged growth rates in all patches.

However, our Theorem~\ref{thm:main} shows that the reverse is true.

\medskip
The paper is organized as follows.  In Section~\ref{section:main} the main concepts are introduced and Theorem~\ref{thm:main} is formulated.  In Section~\ref{section:positive} we give a proof of Theorem~\ref{thm:main} under the assumption that $A$ has all entries positive.  Section~\ref{section:general_nonnegative} deals with a general case.

Vectors [matrices] with nonnegative (resp.\ positive) coordinates [entries] will be refereed to as {\em nonnegative\/} (resp.\ {\em positive\/}) vectors [matrices].

\section{Main concepts}
\label{section:main}
Assume that $A$ is an $N \times N$ nonnegative matrix.

Further, let $D \colon \Omega \to \RR^{N \times N}$ be a measurable matrix function satisfying:
\begin{enumerate}
\item[(A1)]
For each $\omega \in \Omega$, $D(\omega) = \diag(d_{1}(\omega), \dots, d_{N}(\omega))$ with $d_{i}(\omega) > 0$, $1 \le i \le N$;
\item[(A2)]
$\lnplus{\max\limits_{i}d_{i}(\cdot)}$ belongs to $L_1\OFP$.
\end{enumerate}

We consider random matrix systems of the form
\begin{equation}
\label{eq:RMS-D}
\{ A D(\omega) \}_{\omega \in \Omega}.
\end{equation}
We thus have $s_{ij}(\omega) = a_{ij} d_{j}(\omega)$ for all $1 \le i, j \le N$, $\omega \in \Omega$.

As we will be using some results from~\cite{JMSh2}, we introduce here some auxiliary functions from that paper, as well as their properties:
\begin{equation}
\label{eq:mM}
\begin{aligned}
m_{c,i}(\omega) & := \min\limits_{1 \le j \le N} s_{ji}(\omega) = d_{i}(\omega) \cdot \min\limits_{1 \le j \le N} a_{ji},
\\
M_{c,i}(\omega) & := \max\limits_{1 \le j \le N} s_{ji}(\omega) = d_{i}(\omega) \cdot \max\limits_{1 \le j \le N} a_{ji},
\\
M(\omega) & := \max\limits_{1 \le i, j \le N} s_{ij}(\omega) \le \max\limits_{1 \le i \le N} d_{i}(\omega) \cdot \max\limits_{1 \le i, j \le N} a_{ij}.
\end{aligned}
\end{equation}

For $1 \le i \le N$ define $\bar{d}_i$ by
\begin{equation*}
\ln{\bar{d}_i} = \int\limits_{\Omega} \ln{d_i(\cdot)} \, d\mathbf{P}(\cdot),
\end{equation*}
where $\ln{0} = -\infty$, $e^{-\infty} = 0$.  As a consequence of (A2), $\bar{d}_i \in [0, \infty)$.  Let $\bar{D}$ stand for the diagonal matrix $\diag(\bar{d}_1, \dots, \bar{d}_N)$.

The matrix $A \bar{D}$ is nonnegative, so, by the Frobenius--Perron theorem, see~\cite[Thm.\ 1.3.2]{BP}, its spectral radius is an eigenvalue such that an eigenvector corresponding to it can be chosen nonnegative.

The following is the main result of the paper.
\begin{theorem}[Main Theorem]
\label{thm:main}
Under (A1)--(A2) the top Lyapunov exponent of~\eqref{eq:RMS-D} is bounded from below by the logarithm of the spectral radius of $A \bar{D}$.
\end{theorem}
The case of the zero spectral radius of $A \bar{D}$ is obvious, so from now on we assume that the spectral radius of $A \bar{D}$ is positive.

\section{$A$ is a positive matrix}
\label{section:positive}
In the present section we give a proof of Theorem~\ref{thm:main} under the additional assumption that $A$ is a positive matrix.  This allows us to apply the theory of random systems of positive matrices as presented in~\cite{JMSh2}.

The top Lyapunov exponent $\lambda$ can now be expressed as the logarithmic growth rate of some distinguished positive vector.  Indeed, the following result holds.
\begin{proposition}
\label{prop:principal}
There exists a measurable mapping $w = (w_1, \dots, w_N) \colon \allowbreak \Omega_0 \to \RR^{N}$, $w(\omega)$ is a positive vector with $\norm{w(\omega)} = 1$ for all $\omega \in \Omega_0$, where $\theta(\Omega_0) = \Omega_0$, $\mathbf{P}(\Omega_0) = 1$, such that
\begin{equation}
\label{eq:invariance}
S(\omega) w(\omega) = \rho(\omega) w(\theta\omega), \quad \forall \, \omega
\in \Omega_0,
\end{equation}
with $\rho(\omega) > 0$, and
\begin{equation*}
\lim\limits_{n \to \infty} \frac{\ln{\norm{S^{(n)}(\omega) w(\omega)}}}{n} = \lambda \quad \forall \, \omega \in \Omega_0.
\end{equation*}
\end{proposition}
\begin{proof}
By~\cite[Thm.\ 2.3 and Prop.\ 3.2(1)]{JMSh2}, it suffices to show that $\lnplus(\ln{M_{c,i}(\cdot)} - \ln{m_{c,i}(\cdot)})$ belongs to $L_1\OFP$ for all $1 \le i \le N$.  But
\begin{equation*}
\ln{M_{c,i}(\omega)} - \ln{m_{c,i}(\omega)} = \ln{\max\limits_{1 \le j \le N} a_{ji}} - \ln{\min\limits_{1 \le j \le N} a_{ji}},
\end{equation*}
which is a nonnegative constant.
\end{proof}
$\lambda$ is now referred to as the {\em generalized principal Lyapunov exponent\/} of~\eqref{eq:RMS-D}.

It is straightforward from~\eqref{eq:invariance} that
\begin{equation*}
\lambda = \lim\limits_{n \to \infty} \frac{1}{n} \sum\limits_{k = 0}^{n-1} \ln{\rho(\theta^{k}\omega)}, \quad \omega \in \Omega_0.
\end{equation*}

Before we proceed to the proof of Theorem~\ref{thm:main} in the case of positive $A$ we formulate and prove an auxiliary result which will guarantee that Birkhoff's ergodic theorem can be applied.
\begin{lemma}
\label{lm:integrability}
For each $1 \le i \le N$, the function $\ln{w}_i(\cdot)$ is bounded uniformly on $\Omega_0$.
\end{lemma}
\begin{proof}
We apply estimates used in the proof of~\cite[Prop.\ 3.2(1)]{JMSh2}.  Fix $\omega \in \Omega_0$.  Observe that $w(\omega)$ equals $\frac{S(\theta^{-1}\omega) w(\theta^{-1}\omega)}{\rho(\theta^{-1}\omega)}$ with $w(\theta^{-1}\omega)$ a positive vector and $\rho(\theta^{-1}\omega) > 0$, and denote $u(\omega) = (u_1(\omega), \ldots, u_{N}(\omega)) := w(\theta^{-1}\omega)/\rho(\theta^{-1}\omega)$. We have, for $1 \le i \le N$,
\begin{equation*}
\sum\limits_{j = 1}^{N} m_{c,j}(\theta^{-1}\omega) u_j(\omega) \le w_i(\omega) \le \sum\limits_{j = 1}^{N} M_{c,j}(\theta^{-1}\omega) u_j(\omega).
\end{equation*}
When we put
\begin{equation*}
\gamma(\omega) := \sum\limits_{j = 1}^{N} m_{c,j}(\theta^{-1}\omega) u_j(\omega),
\end{equation*}
and
\begin{equation*}
\kappa := \max\limits_{1 \le i \le N}\frac{M_{c,i}(\theta^{-1}\omega)}{m_{c,i}(\theta^{-1}\omega)} \stackrel{\eqref{eq:mM}}{=} \max\limits_{1 \le i \le N} \frac{\max\limits_{1 \le j \le N} a_{ji}}{\min\limits_{1 \le j \le N} a_{ji}},
\end{equation*}
we obtain
\begin{equation*}
\gamma(\omega) \le w_i(\omega) \le \kappa \gamma(\omega).
\end{equation*}
Since $\norm{w(\omega)} = 1$, we have
\begin{equation*}
1 \le N \kappa^2 \gamma^2(\omega),
\end{equation*}
from which it follows that
\begin{equation*}
\ln{w_i(\omega)} \ge \ln{\gamma(\omega)} \ge - \tfrac{1}{2} \ln{N} - \ln{\kappa}.
\end{equation*}
As $\ln{w_i(\omega)} \le 0$, the assertion follows.
\end{proof}
\begin{proof}[Proof of Theorem~\ref{thm:main} for positive $A$]
The functions $\ln{w_i(\cdot)}$ belong to $L_1\OFP$, by Lemma \ref{lm:integrability}, and   the functions $\lnplus{d_i(\cdot)}$ belong to $L_1\OFP$, by (A2).  Birkhoff's ergodic theorem guarantees that for a.e.\ $\omega \in \Omega_0$ the equalities
\begin{gather*}
\lim\limits_{n \to \infty} \frac{1}{n}\sum\limits_{k=0}^{n-1} \ln{w_i(\theta^{k} \omega)} = \int \ln{w_i(\cdot)} \, d\mathbf{P}(\cdot) \ ( {} > -\infty)
\\
\lim\limits_{n \to \infty} \frac{1}{n}\sum\limits_{k=0}^{n-1} \ln{d_i(\theta^{k} \omega)} = \int \ln{d_i(\cdot)} \, d\mathbf{P}(\cdot) \ ( {} = \ln{\bar{d}_i})
\end{gather*}
hold for all $1 \le i \le N$.  Fix such $\omega$, and put $w(n) := w(\theta^{n}\omega)$, $\rho(n) := \rho(\theta^{n}\omega)$, $n = 0, 1, 2, \dots$.

We have
\begin{equation}
\label{eq:auxiliary}
w_{i}(n + 1) = \frac{1}{\rho(n)} \sum\limits_{j=1}^{N} a_{ij} d_{j}(n) w_j(n).
\end{equation}
Put
\begin{equation*}
\tilde{w}_{i}(n) := \exp\Bigl( \frac{1}{n} \sum\limits_{k=0}^{n-1} \ln{w_i(k)}\Bigr), \quad \hat{w}_{i}(n) := \exp\Bigl( \frac{1}{n} \sum\limits_{k=1}^{n} \ln{w_i(k)}\Bigr).
\end{equation*}
Further, let
\begin{equation*}
\tilde{d}_{i}(n) := \exp\Bigl( \frac{1}{n} \sum\limits_{k=0}^{n-1} \ln{d_i(k)}\Bigr), \quad \tilde{\rho}(n) := \exp\Bigl( \frac{1}{n} \sum\limits_{k=0}^{n-1} \ln{\rho(k)}\Bigr).
\end{equation*}

For any $i, j$ an application of the geometric\nobreakdash-\hspace{0pt}arithmetic mean
inequality gives that
\begin{equation*}
\exp\biggl( \frac{1}{n} \sum\limits_{k=0}^{n-1} \ln\Bigl( \frac{1}{\rho(k)} d_j(k) \frac{w_{j}(k)}{w_{i}(k+1)} \Bigr) \biggl) \le \frac{1}{n} \sum\limits_{k=0}^{n-1} \frac{1}{\rho(k)} d_j(k) \frac{w_{j}(k)}{w_{i}(k+1)}.
\end{equation*}
For each $i$, by multiplying the above inequality by $a_{ij}$, $1 \le j \le N$, and adding the resulting inequalities one obtains, after some calculation, that
\begin{multline*}
\frac{1}{\exp\biggl( \frac{1}{n} \sum\limits_{k=0}^{n-1} \ln{\rho(k)} \biggr)} \sum\limits_{j=1}^{N} a_{ij} \exp\biggl( \frac{1}{n} \sum\limits_{k=0}^{n-1} \ln{d_j(k)} \biggr) \frac{\tilde{w}_j(n)}{\hat{w}_i(n)}
\\
\le \frac{1}{n} \sum\limits_{k=0}^{n-1} \frac{1}{\rho(k)} d_j(k) \frac{\sum\limits_{j=1}^{N} a_{ij} w_j(k)}{w_i(k)} \stackrel{\eqref{eq:auxiliary}}{=} 1,
\end{multline*}
that is,
\begin{equation*}
\sum\limits_{j=1}^{N} a_{ij} \tilde{d}_j(n) \tilde{w}_j(n) \le \tilde{\rho}(n) \hat{w}_i(n).
\end{equation*}
Since
\begin{equation*}
\hat{w}_i(n) = \tilde{w}_i(n) \exp\biggl( \frac{1}{n} \ln{\frac{w_i(n)}{w_i(0)}} \biggr),
\end{equation*}
it follows that $\lim\limits_{n\to\infty} \hat{w}_i(n) =\bar{w}_i$ for all $1 \le i \le N$.

We have therefore found a positive vector $\bar{w}$ such that
\begin{equation*}
A \bar{D} \bar{w} \le e^{\lambda} \bar{w},
\end{equation*}
where the inequality is meant to hold coordinatewise. By~\cite[Thm.\ 2.1.11]{BP}, the spectral radius of $A\bar{D}$ does not exceed $e^{\lambda}$, which concludes the proof.
\end{proof}

\begin{remark}
\label{rm:Arnold_et_al}
If we assume additionally that $\lnminus{\min\limits_{i}d_{i}(\cdot)}$ belongs to $L_1\OFP$, then we can use \cite[Thm.\ 3.1]{ArnGundDem} or~\cite[Thm.\ 3.1(3)]{JMSh2} to obtain Proposition~\ref{prop:principal}.
\end{remark}

\section{$A$ is a general nonnegative matrix}
\label{section:general_nonnegative}
Denote
\begin{equation*}
B =
\begin{pmatrix}
1 & \dots & 1
\\
\vdots & \ddots & \vdots
\\
1 & \dots & 1
\end{pmatrix}.
\end{equation*}
For any $\epsilon > 0$ denote by $\lambda_{\epsilon}$ the top Lyapunov exponent of system~\eqref{eq:RMS-D} with $A$ replaced by $A + {\epsilon}B$, that is, of the system
\begin{equation*}
\{ (A + {\epsilon}B) D(\omega) \}_{\omega \in \Omega}.
\end{equation*}
The fact that $\lambda$ equals the limit, as $\epsilon \to 0^{+}$, of $\lambda_{\epsilon}$ is a consequence, for instance, of~\cite[Thm.\ 1]{Fr-GT-Q}.   We will give, however, a much more direct proof here.

It is a standard exercise that for a nonnegative $N$ by $N$ matrix $C$ there holds
\begin{equation*}
\norm{C} = \sup\{\,\norm{Cu}: u = (u_1, \dots, u_N), \ u_i \ge 0, \ \norm{u} = 1 \,\}
\end{equation*}
(cf., e.g., \cite[Lemma 3.1.1]{JMSh-mono}).  Consequently, for any $0 < \epsilon_1 \le \epsilon_2$ we have $\norm{S^{(n)}(\omega)} \le \norm{S_{\epsilon_1}^{(n)}(\omega)} \le \norm{S_{\epsilon_2}^{(n)}(\omega)}$, and, as a result,
\begin{equation*}
\lambda \le \liminf\limits_{\epsilon \to 0^{+}}\lambda_{\epsilon}.
\end{equation*}
On the other hand, it follows from~\eqref{eq:Kingman2} that the top Lyapunov exponent is upper semicontinuous, in~particular, $\lambda \ge \limsup\limits_{\epsilon \to 0^{+}}\lambda_{\epsilon}$.  Therefore
\begin{equation*}
\lambda = \lim\limits_{\epsilon \to 0^{+}}\lambda_{\epsilon}.
\end{equation*}
An analogous reasoning can be repeated for averaged matrices.  Thus we obtain the desired result.

\section*{Concluding remarks}

Analogs of Theorem~\ref{thm:main} for some systems of differential equations have been known for some time.  To the author's knowledge, the first result giving the lower estimate of the principal Lyapunov exponent in~terms of the ``principal Lyapunov exponent'' of the time-averaged system was proved for almost periodic linear parabolic PDEs of second order in~\cite{HuShVi}.

Since that time, many further results of that kind have appeared in the literature.  Their common feature seems to be that the interactions between ``components'' are time\nobreakdash-\hspace{0pt}independent.  For more, see a survey paper~\cite{Mi-surv}.

\section*{Acknowledgements}
I would like to thank the referee for helpful comments.


\begin{thebibliography}{99}

\bibitem{Arn}
\newblock
L. Arnold,
\newblock
\emph{Random Dynamical Systems},
\newblock
Springer Monogr. Math., Springer, Berlin, 1998.

\bibitem{ArnGundDem}
\newblock
L. Arnold, V. M. Gundlach and L. Demetrius,
\newblock
Evolutionary formalism for products of positive random matrices,
\newblock
\emph{Ann. Appl. Probab.}, \textbf{4}(3) (1994), 859--901.

\bibitem{BP}
\newblock
A. Berman and R. J. Plemmons,
\newblock
\emph{Nonnegative Matrices in the Mathematical Sciences}, revised reprint of the 1979 original, Classics Appl. Math., \textbf{9}, SIAM, Philadelphia, PA, 1994.

\bibitem{Fr-GT-Q}
\newblock
G. Froyland, C. Gonz\'alez-Tokman and A. Quas,
\newblock
Stochastic stability of Lyapunov exponents and Oseledets splittings for semi-invertible matrix cocycles,
\newblock
\emph{Comm. Pure Appl. Math.}, \textbf{68}(11) (2015), 2052--2081.


\bibitem{HuShVi}
\newblock
V. Hutson, W. Shen and G. T. Vickers,
\newblock
Estimates for the principal spectrum point for certain time-dependent parabolic operators,
\newblock
\emph{Proc. Amer. Math. Soc.}, \textbf{129}(6) (2001), 1669--1679.

\bibitem{Mi-surv}
\newblock
J. Mierczy\'nski,
\newblock
Estimates for principal Lyapunov exponents: A survey,
\newblock
\emph{Nonauton. Dyn. Syst.},  \textbf{1}(1) (2014), 137--162.  Available at arXiv: 1406.0992.

\bibitem{JMSh-mono}
\newblock
J. Mierczy\'nski and W. Shen,
\newblock
\emph{Spectral Theory for Random and Nonautonomous Parabolic Equations and Applications},
\newblock
Chapman Hall/CRC Monogr. Surv. Pure Appl. Math., Chapman \& Hall/CRC, Boca Raton, FL, 2008.

\bibitem{JMSh2}
\newblock
J. Mierczy\'nski and W. Shen,
\newblock
Principal Lyapunov exponents and principal Floquet spaces of positive random dynamical systems. II. Finite-dimensional case,
\newblock
\emph{J. Math. Anal. Appl.}, \textbf{404}(2) (2013), 438--458.

\bibitem{Schreib}
\newblock
S. J. Schreiber,
\newblock
Interactive effects of temporal correlations, spatial heterogeneity and dispersal on population persistence,
\newblock
\emph{Proc. Roy. Soc. Edinburgh Sect. B}, \textbf{277} (2010), 1907--1914.

\end{thebibliography}
\end{document}